\documentclass{article}
\usepackage[utf8]{inputenc}
\usepackage{amsmath}
\usepackage{amssymb}
\usepackage{amsthm}
\usepackage{bbm}
\usepackage{graphicx}
\usepackage{dsfont}
\usepackage[margin=1.5in]{geometry}

\usepackage{enumitem}
\usepackage{xcolor}

\newtheorem{lemma}{Lemma}[section]
\newtheorem{proposition}[lemma]{Proposition}
\newtheorem{remark}[lemma]{Remark}
\newtheorem{theorem}[lemma]{Theorem}

\newtheorem{corollary}[lemma]{Corollary}

\newcommand{\cA}{\mathcal{A}}
\newcommand{\cS}{\mathcal{S}}
\newcommand{\cD}{\mathcal{D}}

\newcommand{\cH}{\mathcal{H}}
\newcommand{\R}{\mathbb{R}}
\newcommand{\Z}{\mathbb{Z}}
\newcommand{\CP}{\mathbb{C}\mathbb{P}}
\newcommand{\RP}{\mathbb{R}\mathbb{P}}
\newcommand{\bK}{{\mathbb{K}}}
\newcommand{\Id}{\mathds{1}}
\DeclareMathOperator{\area}{Area_{_\omega}}
\DeclareMathOperator{\Spec}{\mathrm{Spec}}

\newenvironment{Properties}
{\begin{list}{}{
\setlength{\topsep}{6pt}%
\setlength{\itemsep}{4pt}%
\setlength{\labelsep}{0pt}%
\setlength{\leftmargin}{0pt}%
\setlength{\labelwidth}{0pt}%
\setlength{\listparindent}{0pt}}%
\setlength{\parskip}{0pt}%
}
{\end{list}}

\title{Pseudo-holomorphic triangles and the median quasi-state}
\author{Pazit Haim-Kislev, Michael Khanevsky, Asaf Kislev, Daniel Rosen}
\date{}

\begin{document}

\maketitle

\begin{abstract}
    We prove that the minimal area of a holomorphic triangle whose boundary lies on the union of any three Lagrangian submanifolds is bounded from above by the Lagrangian spectral norm of any pair taken out of the three. We show a relation between this result and the median quasi-state on the 2-sphere. The median quasi-state gives rise to a measure of Poisson non-commutativity of any pair of functions. We answer a question of Entov--Polterovich--Zapolsky by giving a sharp upper bound on the ratio between this measure and the uniform norm of the Poisson bracket of the pair. The sharpness of this bound also implies a lower bound on the defect of the Calabi quasi-morphism.
\end{abstract}

\section{Introduction and main results}

\subsection{Bounds on holomorphic triangles}

The theory of $J$-holomorphic curves was introduced into symplectic geometry in a seminal paper of Gromov in 1985 \cite{gromov}. Since then it has profoundly influenced the study of symplectic geometry with many key results being dependent upon it, including the creation of Floer homology. We refer to \cite{McDuffSalamonJHol} for a thorough review on $J$-holomorphic curves.
In this work we study $J$-holomorphic triangles: $J$ holomorphic curves with boundaries on three Lagrangian submanifolds.
We introduce a new invariant of a pair of Lagrangian submanifolds based on the largest minimal symplectic area of a $J$-holomorphic triangle over all choices of a third Lagrangian submanifold. We show that this invariant is bounded from above by the Lagrangian spectral norm, and discuss applications.

First let us briefly review the settings in which we will be working.
Let $(M^{2n},\omega)$ be a closed symplectic manifold.
Fix a closed Lagrangian submanifold $L \subset M$.
Recall that a Lagrangian submanifold $L$ is said to be monotone (see e.g. \cite{McDuffSalamon}) if the symplectic area homomorphism
\[ \omega : \pi_2(M,L) \to \R,\]
and the Maslov index
\[ \mu : \pi_2(M,L) \to \Z\]
satisfy that there exists a positive constant $\tau$ so that $\omega(A) = \tau \mu(A)$ for each $A \in \pi_2(M,L)$.
The positive generator of $\text{Im}(\mu)$ is called the minimal Maslov number $N_L$ of $L$, we assume throughout the paper that $N_L \geq 2$.

The Lagrangian spectral invariant associated to $L$, following \cite{LeclercqZapolsky}, is a function
\[ c : QH_*(L) \times C^\infty(M\times [0,1]) \to \R \cup \{-\infty\}.\]
We refer to \cite{LeclercqZapolsky} for many properties and applications. We will briefly recall the definition and some of the properties in Section \ref{preliminariesSection} below.
The ``spectral norm" associated to an idempotent $a \in QH_*(L)$ is defined by
\[ \nu(L;a,H) = c(L;a,H) + c(L;a,\overline{H}),\]
where $\overline{H}_t := -H_{1-t}$. 
We define the spectral norm of a pair of Hamiltonian isotopic Lagrangian submanifolds by
\[ \nu(L_1,L_2;a) := \inf_H \{\nu(L_1;a,H)\},\]
where the infimum is taken over Hamilfonian functions $H$ with $\phi_H^1(L_2) = L_1$, where $\phi_H^1$ is the time-1-map.
The Lagrangian spectral norm was first introduced by Viterbo in \cite{Viterbo-specGF} for the case where $L_1$ is the zero-section of a cotangent bundle.
We remark that the spectral norm is non-negative whenever $a$ is an idempotent (see e.g. the proof of Theorem 35 in \cite{LeclercqZapolsky}), and moreover, it is also non-degenerate. We refer to \cite{KislevShelukhin} for the proof in the most general case, as well as for a review of the literature. 

Let us fix a regular compatible almost complex structure $J$.
Let $L_1, L_2, L_3 \subset M$ be closed Hamiltonian isotopic Lagrangian submanifolds and intersecting transversely.
Denote by $\cD(L_1,L_2,L_3;J)$ the set of immersed $J$-holomorphic triangles with edges lying on $L_1,L_2,L_3$.
Denote 
\begin{equation}
    \cS(L_1,L_2) = \sup_{L_3,J} \inf_{\Delta \in \cD(L_1,L_2,L_3;J)} \area(\Delta) ,
\end{equation}
where the supremum is taken over all Lagrangian submanifolds $L_3$ Hamiltonian isotopic to $L_1$ that intersects $L_1,L_2$ transversely, and $J$ a regular compatible almost complex structure.

We are now ready to state our first result.
\begin{theorem}
\label{minTriangle}
Let $L_1,L_2 \subset M$ be closed, monotone and Hamiltonian isotopic Lagrangian submanifolds intersecting transversely.
For any idempotent $a \in QH_*(L_1)$,
\[ \cS(L_1,L_2) \leq \frac{1}{2} \nu(L_1,L_2;a).\]
\end{theorem}

\begin{remark}
The geometric meaning behind Theorem \ref{minTriangle} is as follows.
For every three Hamiltonian isotopic Lagrangians as above, there exists a pseudo-holomorphic triangle between them whose area is at most half of the spectral norm of any Hamiltonian diffeomorphism $\phi$, which maps one of the Lagrangians to another. 
\end{remark}

\begin{remark}
    Theorem \ref{minTriangle} gives a different proof for the fact that in our settings $\nu(L_1,L_2,[L_1])$ is non-degenerate. 
    Indeed, from transversality and compactness, the number of intersection points is finite. Hence for a choice of $L_3$ such that the intersections are transverse and there are no triple intersection points, the number of $J$-holomorphic triangles (up to recappings) is finite, and the smallest area of a $J$-holomorphic triangle has positive area.
\end{remark}

\vspace{10pt}

Although it is in general very hard to compute $\nu(L_1,L_2;a)$, in some cases there exists a global upper bound. In particular, following \cite{KislevShelukhin}, in $M = \CP^n$, and $L = \RP^n$, by taking $a = [L]$ one has
\[ \nu(L_1,L_2;a) \leq \frac{n}{2n+2}.\]

\begin{corollary}
Suppose that for $a \in QH_*(L)$ there is a global bound $\nu(L;a,\cdot)\leq C(a,L)$. Then for any $L_1, L_2$ Hamiltonian isotopic to $L$ one has $\cS(L_1,L_2) \leq \frac{1}{2} C(a,L)$.

In particular, in $M = \CP^n$, and $L = \RP^n$, when taking $a = [L]$ one gets
\[ \cS(L_1,L_2) \leq \frac{n}{4n+4}.\]
\end{corollary}

Next we discuss two applications of Theorem \ref{minTriangle}.

\subsection{Connection with the Lagrangian Hofer distance}

Recall the Hofer norm of a Hamiltonian diffeomorphism
$$ \| \phi \|_{\text{Hofer}} = \inf \left\{ \int_0^1 \left( \max_{x \in M} H_t(x) - \min_{x \in M} H_t(x) \right) dt : \phi_H = \phi \right\} ,$$
where the infimum is taken over all Hamiltonian functions generating $\phi$. It is a deep fact that the Hofer norm is indeed non-degenerate. This was shown by Hofer \cite{hofer} for $M=\R^{2n}$, then generalized by Polterovich \cite{polt}, and finally proven in full generality by Lalonde and McDuff \cite{LalondeMcDuff}.
It is worth mentioning that the existence of such a metric for non-compact groups of transformations is highly unusual.
The Hofer distance was generalized by Chekanov \cite{chekanov} to a distance between two closed Hamiltonian isotopic Lagrangian submanifolds:
$$ d(L_1,L_2) = \inf_{\phi(L_1) = L_2, \; \phi \in Ham(M)} \| \phi \|_{\text{Hofer}}. $$

It was shown in \cite{LeclercqZapolsky} that when taking $a = [L]\in QH_*(L)$ one has
$ d(L_1,L_2) \geq \nu(L_1,L_2;[L_1]),$ which gives the following.
$$ d(L_1,L_2) \geq 2 \cS(L_1,L_2). $$    
We note that this observation introduces a new lower bound for the Hofer distance which is of a geometric nature (cf. the lower bound in \cite[Corollary 3.13]{BarraudCornea}).  

\subsection{The median quasi-state on $S^2$.}
Let $M$ be a closed monotone symplectic manifold with bounded spectral norm.
Entov and Polterovich \cite{EntovPolt} introduced the notion of symplectic quasi-states $$\zeta : C^\infty(M) \to \R,$$ showed their existence using Floer theoretic tools, and derived remarkable applications. 

In \cite{EntovPoltZap}, it is shown that there exists $C > 0$ with
\begin{equation}
\label{epzInequality}
| \zeta(F+G) - \zeta(F) - \zeta(G) |^2 \leq C \| \{F,G\} \| , 
\end{equation}
where $\{F,G\}$ is the Poisson bracket of $F$ and $G$.
It is also shown that $C \leq 2D$, where $D$ is the defect (defined {\it ibid.}) of the Calabi quasi-morphism \cite{EntovPoltQuasiMorphisms}.
As a consequence, for $M = S^2$ this gives the bound $C \leq 4$.
In \cite{Zapolsky}, this bound is improved to $C \leq 1$.
In fact, in \cite{Zapolsky} it is shown that $$| \zeta(F+G) - \zeta(F) - \zeta(G) |^2 \leq  \| \{F,G\} \|_{L^1} ,$$ and moreover, in \cite{AnatAmir} it is shown that this last inequality is sharp. 
In \cite{EntovPoltZap}, the authors pose the question of finding the sharp value of the constant $C$ in \eqref{epzInequality}. Below we prove that the sharp value is $\frac{1}{4}$ (see Theorem \ref{median_theorem}).

Let us recall the description of the symplectic quasi-state in $S^2$ as the median of the Reeb graph. Here we are using the fact \cite{EntovPolt} that in $S^2$ two symplectic quasi-states which vanish on functions with displaceable support coincide, and specifically the median quasi-state is equal to the Floer-homological quasi-state.
Let $F : S^2 \to \R$ be a smooth Morse function. Denote by $\Gamma_F$ the set of connected components of level sets of $F$, and denote by $\pi_F : S^2 \to \Gamma_F$ the natural projection. $\Gamma_F$ has a natural structure of a graph, called the Reeb graph, whose vertices correspond to level sets of critical points of $F$.  Normalize $\omega$ such that the area of $S^2$ is $1$, and then the measure $(\pi_F)_* \omega$ defines a Borel probability measure on $\Gamma_F$ whose restriction to any open edge is homeomorphic to the Lebesgue measure on an open interval. Hence, following \cite{EntovPoltQuasiMorphisms}, there exists a unique point $m_{\Gamma_F}$  in $\Gamma_F$ such that every connected component of $\Gamma_F \setminus m_{\Gamma_F}$ has area less than $\frac{1}{2}.$ Finally let $\zeta(F) := F(\pi_F^{-1}(m_{\Gamma_F}))$.
We note that $\pi_F^{-1}(m_{\Gamma_F})$ is generically either an equator, an embedded circle that divides $S^2$ into two parts of area $\frac{1}{2}$, or a figure eight, a closed curve with a single self-intersection that divides $S^2$ into three parts where each part has area less than $\frac{1}{2}$.
We refer to \cite{PoltRosenBook} for a thorough overview on symplectic quasi-states and function theory on symplectic manifolds in general.

\begin{theorem}
\label{median_theorem}
For $F,G \in C^\infty(S^2)$ one has
$$
\sup  \frac{| \zeta(F+G) - \zeta(F) - \zeta(G) |^2}{\| \{F,G\}\|}  = \frac{1}{4} .
$$
\end{theorem}

\begin{remark}
In the proof one uses a bound on the area of a triangle between $\pi_F^{-1}(m_{\Gamma_F}), \pi_G^{-1}(m_{\Gamma_G})$ and $\pi_{F+G}^{-1}(m_{\Gamma_{F+G}})$ (cf. Section \ref{medianSection} below). If all three of them are equators, Theorem \ref{minTriangle} gives the required bound. However, if some of them are figure eights, one needs to generalize the bound in Theorem \ref{minTriangle} to this case as well which is done in Section \ref{trianglesS2Section} below. It would be interesting to find Floer theoretical methods that would enable bounding the areas of triangles between figure eights.
\end{remark}

\begin{remark}
 Theorem \ref{median_theorem} implies the following lower bound on the defect of the Calabi quasi-morphism $D \geq \frac{1}{8}$.
In $\cite{KislevShelukhin}$ the defect is shown to be bounded above by $\frac{1}{2}$, so overall $\frac{1}{8} \leq D \leq \frac{1}{2}$.
It would be interesting to find the exact value of $D$.
\end{remark}

\subsection{Organization of the paper}
In Section \ref{preliminariesSection} we recall two definitions of Lagrangian Floer homology and show their equivalence, and then we discuss briefly Lagrangian spectral invariants. In Section \ref{holomorphicTrianglesSection} we prove Theorem \ref{minTriangle}. In Section \ref{trianglesS2Section} we bound the area of triangles in $S^2$ between three medians (which might include e.g. figure eights). Finally in section \ref{medianSection} we prove Theorem \ref{median_theorem}.

\subsection{Acknowledgements}
We wish to thank Leonid Polterovich, Yaron Ostrover, and Egor Shelukhin for helpful conversations and suggestions. P. H-K. is partially supported
by the ISF grant No. 938/22. M.K. was partially supported by the Azrieli Fellowship during some of the work on this project.

\section{Preliminaries}
\label{preliminariesSection}

\subsection{Lagrangian Floer homology}

Let $L \subset M$ be a monotone Lagrangian with $N_L \geq 2$. In this text we use two equivalent ways to define the Floer complex of $L$. In each one of them a different set of data is needed for the definition. In the first approach the generators are intersection points with a different Lagrangian submanifold $L'$ Hamiltonian isotopic to $L$. In the second approach the generators are Hamiltonian chords from $L$ to itself, of a time-dependent Hamiltonian $H$. (Actually, there is a unified approach where one considers a pair of Lagrangian submanifolds and a Hamilfonian function, but we will not describe this approach here.) The first approach is needed for us because the $J$-holomorphic triangles that appear in $\cS$, appear naturally in the product structure, and the area of such triangles can be bounded by action difference of intersection points. The second approach is needed for us because the spectral norm is defined more naturally in this language. In this subsection we give the definitions of both approaches and explain the equivalence between them. We refer to \cite{oh1, oh2, LeclercqZapolsky, KislevShelukhin, usher} and references therein for more details.

\subsubsection{Lagrangian intersections}

Let us define the Lagrangian Floer complex in the first approach for a pair of Lagrangian submanifolds $L_1,L_2$. Fix a path $\gamma_0 : [0,1] \to M$ from $L_1$ to $L_2$ (i.e. $\gamma_0(0)\in L_1$ and $\gamma_0(1) \in L_2$).
Let $\Gamma(L_1,L_2;\gamma_0)$ be the space of paths from $L_1$ to $L_2$ homotopic with endpoints on $L_1$ and $L_2$ respectively to $\gamma_0$. We will sometimes omit $\gamma_0$ from the notations and write simply $\Gamma(L_1,L_2)$.
Let $\gamma \in \Gamma(L_1,L_2)$. Let $\tilde{\gamma}$ be a path from $\gamma_0$ to $\gamma$ in $\Gamma(L_1,L_2)$. i.e.
\[ \tilde{\gamma}(0,\cdot) = \gamma_0(\cdot) ,\]
\[ \tilde{\gamma}(1,\cdot) = \gamma(\cdot) ,\]
\[ \tilde{\gamma}(\cdot,0) \in L_1 ,\]
\[ \tilde{\gamma}(\cdot,1) \in L_2 .\]
We call $\tilde{\gamma}$ a capping of $\gamma$. Given another capping $\tilde{\gamma}'$, one can consider the concatenation $\bar{\tilde{\gamma}} \# \tilde{\gamma}'$ as a map from the cylinder to $M$ with the boundary on $L_1$ and $L_2$. To this cylinder one can attach its symplectic area and its Maslov index. We say that $\tilde{\gamma}$ and $\tilde{\gamma}'$ are equivalent if both the symplectic area and the Maslov index of the cylinder vanish.
Let $\tilde{\Gamma}(L_1,L_2)$ be a covering of $\Gamma(L_1,L_2)$ composed of pairs $(\gamma,[\tilde{\gamma}])$ where $[\tilde{\gamma}]$ is an equivalence class.
The action $\cA(\gamma,[\tilde{\gamma}])$ equals minus the symplectic area of $\tilde{\gamma}$. 
One can check that the critical points of the action are constant paths in $L_1 \cap L_2$ (with different cappings). The Floer chain complex $CF_*(L_1,L_2)$ is generated (over $\Z_2$) by critical points of the action, i.e. intersection points with different cappings. 


The boundary operator of this complex counts finite energy $J$-holomorphic curves $u: \R \times [0,1] \to M$ that satisfy
\[ u(\cdot,0) \in L_0,\]
\[ u(\cdot,1) \in L_1,\]
\[ \lim_{t \to \pm \infty}u(t,\cdot) = x_\pm ,\]
\[ [\tilde{x}_- \# u]  =  [\tilde{x}_+].\]
We call such $u$ Floer strips connecting $(x_\pm, \tilde{x}_\pm)$.

Finally, let us describe the product structure. Let $\tilde{x}_{12} \in CF_*(L_1,L_2)$ and $\tilde{x}_{23} \in CF_*(L_2,L_3)$ be intersection points with cappings. We want to define the product $\tilde{x}_{12} \ast \tilde{x}_{23} \in CF_*(L_1,L_3)$ and extend this definition by linearity. In order to define this product, one needs to be careful when choosing $\gamma_0$ 
for each pair of Lagrangian submanifolds. One way to do this is to pick a point $y \in M$ 
, and then consider pairs $(L, \gamma_L)$, where $\gamma_L$ connects $y$ to some point $x_L \in L$. 
One can concatenate these paths to get base paths $\gamma_{ij} := \gamma_{L_i} \# \overline{\gamma}_{L_j}$.



Fix now $\tilde{x}_{13} \in CF_*(L_1, L_3)$. The coefficient $\langle \tilde{x}_{12} \ast  \tilde{x}_{23},  \tilde{x}_{13} \rangle$ is defined by counting $J$-holomorphic triangle with boundaries on $L_1 \cup L_2 \cup L_3$ between $x_{12}$, $x_{23}$ and $x_{13}$ which are compatible with the given cappings. Slightly more accurately, one counts $J$-holomorphic maps 
\[ u : D^2 \setminus \{z_1,z_2,z_3\} \to M ,\]
where $z_1,z_2,z_3 \in \partial D^2$, such that the three parts of the boundary lie on $L_1,L_2,L_3$ respectively, $u$ extends continuously to $D^2$ by $u(z_1) = x_{12}, u(z_2) = x_{23}, u(z_3) = x_{13}$, and $[\tilde{x}_{12} \#\tilde{x}_{23} \# u] = [\tilde{x}_{13}]$. Note that 
$$\int_{D^2} u^* \omega = \mathcal{A}(\tilde{x}_{12}) + \mathcal{A}(\tilde{x}_{23}) - \mathcal{A} (\tilde{x}_{13}).$$

\subsubsection{Hamiltonian chords}

Let us describe the second approach to Lagrangian Floer theory. Fix a closed monotone Lagrangian $L \subset M$. Let $H : M \times [0,1] \to \R$ be a Hamiltonian function which satisfies $H(0)=H(1)=0$, and let $J_t$ be a time-dependent almost complex structure. Denote $\Omega(L) = \{ \alpha : [0,1] \to M : \alpha(0), \alpha(1) \in L, [\alpha] = 0 \in \pi_1(M,L) \}$.
Fix a point $x_0 \in L$. A capping of a curve $\alpha$ is a homotopy in $\Omega$ from $\alpha$ to the constant path $x_0$. i.e. $\tilde{\alpha} : [0,1]^2 \to M$ with
\[ \tilde{\alpha}(0,t) = x_0,\]
\[  \tilde{\alpha}(1,t) = \alpha(t) ,\]
\[  \tilde{\alpha}(s,0), \tilde{\alpha}(s,1) \in L.\]
Two such homotopies are equivalent if the area and Maslov index of their concatenation is zero.
Let $\widetilde{\Omega}(L)$ be the covering space of $\Omega(L)$ composed of pairs $(\alpha, [\tilde{\alpha}])$, where $[ \tilde{\alpha}]$ is an equivalence class.
The action of a pair $(\alpha, [ \tilde{\alpha}])$ is
\[ \cA(\alpha,[ \tilde{\alpha}]) = \int_0^1 H_t(\alpha(t)) dt - \area( \tilde{\alpha}) .\]
The critical points of $\cA$ on $\widetilde{\Omega}(L)$ are Hamiltonian chords of $H$ with different cappings.
The Floer chain complex is generated (over $\Z_2$) by the critical points, and we extend the definition of the action by $\cA( \sum_{i=1}^N (\alpha_i,[\tilde{\alpha}_i])) = \max_{i=1}^N \cA(\alpha_i,[\tilde{\alpha}_i]).$
The grading of the complex is given by the Conley--Zehnder index.
The boundary operator counts isolated finite energy solutions $v : \R \times [0,1] \to M$ to the Floer equation
\[ \partial_s v + J_t (\partial_t v - X_{H_t}(v)) = 0 , \]
with boundary conditions 
\[v(\cdot,0),v(\cdot,1) \in L,\]
\[v(\pm \infty,\cdot) = \alpha_\pm(\cdot),\]
\[ [\tilde{\alpha}_- \# v] = [\tilde{\alpha}_+].\]

Let us describe the product structure.
Let $\tilde{y}_1 \in CF(L,H_1)$,$\tilde{y}_2 \in CF(L,H_2)$, where we denote $\tilde{y}_i = (\alpha_i,[\tilde{\alpha}_i])$. Choose $H_1,H_2$ that satisfy $(H_1)_t=(H_2)_t=0$ near $t=0$ and near $t=1$.
The product $\tilde{y}_1 \ast \tilde{y}_2 \in CF(L, H_1 \# H_2)$ is defined by counting isolated maps $v \in D^2 \setminus\{z_1,z_2,z_3\}$ satisfying the following conditions.
For $z \in D^2 \setminus\{z_1,z_2,z_3\}$ let $J_z$ be an almost complex structure on $M$ compatible with $\omega$ depending on $z$, and let $H_\xi : M \to \R$ be a Hamiltonian one form depending on a vector $\xi \in T (D^2 \setminus\{z_1,z_2,z_3\})$. We require $J_z$ and $H_\xi$ to be compatible with the Floer data near the punctures. Specifically we choose cylindrical strips connecting $z_1$ and $z_3$ and $z_2$ and $z_3$ respectively, and we take $H_\xi$ to be  $(H_1)_t dt$ and $(H_2)_t dt$ on both cylindrical strips respectively and extend by zero (see \cite[Section 2.2]{LeclercqZapolsky} and \cite[Section 2.5]{KislevShelukhin}).
Then $v$ is required to solve the following equation
\begin{equation} \label{perturbed-j-hol-eq} (dv(\xi) - X_{H_{\xi}}(v))^{(0,1)} = 0, \end{equation}
i.e. the complex anti-linear part of the  the 1-form $dv(\xi) - X_{H_{\xi}}$ on $D^2 \setminus\{z_1,z_2,z_3\}$ with values in the complex vector bundle $(v^* TM, J_z)$ vanishes.
In addition $v |_{\partial D^2} \subset L$ and $v$ is required to be asymptotic to $\alpha_1,\alpha_2,\alpha_3$ along the punctures, and satisfy $[\tilde{\alpha}_1 \# \tilde{\alpha}_2 \# v] = [\tilde{\alpha}_3]$.

\subsubsection{The equivalence of both approaches}

Finally, let us describe the equivalence between the two approaches.
Let $L_1,L_2$ be closed monotone Hamiltonian isotopic intersecting transversely Lagrangian submanifolds with $N_{L_1} \geq 2$.
Let $H$ be a Hamiltonian function with $\phi_H^1(L_2) = L_1$, so that there exists $y \in L_1$ with 
\begin{equation}\label{eq-normalization-zero-int}
\int_0^1 H_t(\phi_H^t(y)) dt = 0.    
\end{equation}
In order to define $CF(L_1,L_2)$ we choose the fixed path $\gamma_0$ to be $(\phi_{H}^t)^{-1}(y)$.
In order to define $CF(L_1,{H})$ we choose the basepoint $y  \in L_1$.

Consider the map $\Phi_H : \Gamma(L_1, L_2) \to \Omega (L_1, H)$ given by $\gamma \mapsto \phi_H^t\gamma (t)$. It lifts in an obvious way to a map we still denote by $\Phi_H : \widetilde{\Gamma}(L_1, L_2) \to \widetilde{\Omega} (L_1, H)$, and one easily checks that in view of condition \ref{eq-normalization-zero-int} it preserves action. In particular, $\Phi_H$ induces a bijection between generators of the Floer complexes $CF(L_1, L_2)$ and $CF(L_1, H)$.

Next, let us consider the differentials on both complexes. Let $(x_\pm, \tilde{x}_\pm)$ be two intersection points with cappings. Given a $J$-holomorphic strip $u$ between $(x_\pm, \tilde{x}_\pm)$, consider the map $v(s,t) = \phi_H^t u(s,t)$. It is easy to versify that $v$ is a connecting Floer strip (with the conjugated almost complex structure $(\phi_H^t)_* J (\phi_H^t)_*^{-1}$) between the capped Hamiltonian chords $\Phi_H(x_\pm, \tilde{x}_\pm)$. We conclude that $\Phi_H$ intertwines the differentials on the two Floer complexes.

Finally, in terms of the product structure, consider the product $\tilde{x}_{12} * \tilde{x}_{23}$ like before. We will show equivalence to the product of some $\tilde{y}_1 \in CF(L_1,H_{12})$ and $\tilde{y}_2 \in CF(L_1,H_{23} \circ \phi_{12}^{-1})$, where $\phi^t_{12}, \phi^t_{23}$ are the Hamiltonian flows of $H_{12}$ and $H_{23}$ respectively, and $H_{12}, H_{23}$ are such that $\phi_{12}(L_2) = L_1$ and $\phi_{23}(L_3) = L_2$. In addition, choose $H_{12}$ and $H_{23}$ such that they vanish near $t=0$ and near $t=1$. 
Note that in general there is an obvious correspondence between $CF(L,H)$ and $CF(\phi(L), \phi^* H = H \circ \phi^{-1})$ and hence we remark that the choice for the second Hamiltonian $H_{23} \circ \phi_{12}^{-1}$ is consistent with the correspondence that we have showed previously and indeed after applying $\phi_{12}$ it corresponds to $CF(L_2,L_3)$. 
For a $J$-holomorphic triangle $u : D^2 \setminus \{z_1,z_2,z_3\}$ one can define $v: D^2 \setminus \{z_1,z_2,z_3\}$ by $v(z) = \phi_{z} u(z)$, where $\phi_z$ is defined as follows. 
Denote the cylindrical ends between $z_1$ and $z_3$ and between $z_2$ and $z_3$ by $\epsilon_{13} : \R \times [0,1] \to D^2$ and $\epsilon_{23} : \R \times [0,1] \to D^2$ respectively. Note that $Im(\epsilon_{13})$ and $Im(\epsilon_{23})$ divide $D^2 \setminus(Im(\epsilon_{13}) \cup Im(\epsilon_{23}))$ into three parts denoted by $\Sigma_1, \Sigma_2, \Sigma_3$. Define $\phi_z = \Id$ for $z \in \Sigma_1$, $\phi_z = \phi_{12}$ for $z \in \Sigma_2$ and $\phi_z = \phi_{12} \circ \phi_{23}$ for $z \in \Sigma_3$. For $z= \epsilon_{13}(s,t)$ set $\phi_z = \phi_{12}^t$, and for $z = \epsilon_{23}(s,t)$ we set $\phi_z = \phi_{12}^1 \circ \phi_{23}^t$.
It is not hard to check that now $v$ satisfies \eqref{perturbed-j-hol-eq} for $J_{z} := (\phi_{z})_* J (\phi_{z})_*^{-1}$. 
In terms of the base point and base paths, choose a base point $y \in L_1$ for $CF(L_1,H_{ij})$, and choose $\gamma_{L_1} \equiv y, \gamma_{L_2} = \{\phi_{12}^{-t}(y)\},$ and $\gamma_{L_3} = \gamma_{L_2} \# \{ \phi_{23}^{-t} \circ \phi_{12}^{-1}(y)\}$.

This again demonstrates the one-to-one correspondence between the two counting problems in the two approaches for Floer homology.

\subsection{Spectral invariants}
Spectral invariants were introduced into symplectic topology by Viterbo \cite{Viterbo-specGF} in the context of generating functions, by Schwarz \cite{Schwarz:action-spectrum} and Oh \cite{oh05} in Hamiltonian Floer theory (see also \cite{Usher-spec}), and by Leclercq--Zapolsky \cite{LeclercqZapolsky} in Lagrangian Floer theory, where one can find a comprehensive review of the literature.

Let $HF(L,H)^t$ be the action filtered Floer homology generated by elements in $CF(L,H)$ with actions smaller than $t$.
The inclusion $CF(L,H)^t \subset CF(L,H)$ gives rise to the maps $i_t : HF(L,H)^{t} \to HF(L,H)$. 
Let $\Phi_{PSS} : QH(L) \to HF(L,H)$ be the Piunikhin-Salamon-Schwarz isomorphism (see \cite{albers, biranCornea1,biranCornea2,biranCornea3,katicMilinkovic,ohZhu,PSS}), and define the spectral invariant $c(-,H): QH(L) \setminus\{0\} \to \R$ to be
$$ c(a,H) = \inf \{t : \Phi_{PSS} (a) \in \text{Im}(i_t) \}.$$

The following proposition summarizing the properties of spectral invariants in Lagrangian Floer homology is proved in \cite{LeclercqZapolsky}.

\medskip

\begin{proposition}\label{prop:main_properties_Lagr_sp_invts}
	Let $L$ be a closed monotone Lagrangian of $(M,\omega)$ with minimal Maslov number $N_L \geq 2$. The spectral invariant
	$$c : QH_*(L) \setminus \{0\} \times C^0 \big(M\times [0,1] \big) \rightarrow \R$$
	satisfies the following properties. 
	\begin{Properties} 	
		\item[Spectrality] For $H \in C^\infty \big(M \times [0,1]\big)$, $c(a;H) \in  \Spec(H,L)$.
		\item[Ring action] For $r \in R$, $c(r \cdot a;H) \leq c(a;H)$. In particular, if $r$ is invertible, then $c(r \cdot a;H) = c(a;H)$. If $R = \bK$ is a field, $c(\lambda \cdot a; H) = c(a;H) - \nu(\lambda),$ for all $\lambda \in \Lambda.$
		\item[Symplectic invariance] Let $\psi \in \text{Symp}(M,\omega)$ and $L'=\psi(L)$. Let
		$$c' : QH_*(L') \times C^0\big(M\times [0,1] \big) \to \R$$
		be the corresponding spectral invariant. Then $c(a;H) = c'(\psi_*(a);H \circ \psi^{-1})$. 
		\item[Normalization] If $\alpha$ is a function of time then $$c(a;H+\alpha)=c(a;H) + \int_0^1 \alpha(t) \,dt\,.$$ We have $c(a;0) = \cA(a)$ and $c_+(0)=0$.
		\item[Continuity] For any $H$ and $K$, and $a \neq 0$:
		$$\int_0^1 \min_M (K_t - H_t) \,dt \leq c(a;K) - c(a;H) \leq \int_0^1 \max_M (K_t - H_t) \,dt \,.$$ 
		\item[Monotonicity] If $H \leq K$, then $c(a;H) \leq c(a;K)$.
		\item[Triangle inequality] For all $a$ and $b$, $c(a \ast b; H \# K) \leq c(b;H) + c(a;K)$.
		
		\item[Lagrangian control] If for all $t$, $H_t|_L = c(t) \in \mathbb{R}$ (respectively $\leq$, $\geq$), then
		$$c([L];H)=\int_0^1 c(t) \,dt\quad\text{(respectively } \leq, \geq)\,.$$
		Thus for all $H \in \cH$:
		$$\int_0^1 \min_L H_t \,dt \leq c([L];H) \leq \int_0^1 \max_L H_t \,dt \,.$$
		
		\item[Duality] Let $ a^\vee \in QH^{n-k}(L)$ be the element corresponding to $a \in QH_k(L)$  under the natural duality isomorphism. Then we have
		$$c(a;\overline{H}) = - \inf \big\{c(b;H)\,|\, b \in QH_{n-k}(L) \text{ with } \langle a^\vee, b\rangle \neq 0 \big\}\,.$$
		
		\item[Non-negativity] $c(a;H) + c(a;\overline{H}) \geq 0$ for an idempotent $a \in QH_*(L)$. 
		\item[Maximum] $c(a;H) \leq c([L];H)+\cA(a)$. 
	\end{Properties}
\end{proposition}

\section{Pseudo-holomorphic triangles}
\label{holomorphicTrianglesSection}



The aim of this section is to prove Theorem \ref{minTriangle}.

Let $L_1,L_2,L_3$ be three pairwise Hamiltonian isotopic monotone Lagrangian submanifolds with $N_{L_1} \geq 2$. Pick Hamiltonian functions $H_{ij},i, j \in \{1,2,3\}$ with $H_{ii} = 0$, generating Hamiltonian flows $\phi_{ij} = \{\phi_{ij}^t\}$ with $\phi_{ij}^1(L_j) = L_i$, and $[\phi_{ij} \# \phi_{jk}] = [\phi_{ik}]$ for any $i,j,k \in \{1,2,3\}$. By $\phi_{ij} \# \phi_{jk}$ we mean the concatenation $\phi_{ij}^t$ for half the time, and $\phi_{ij}^1 \circ \phi_{jk}^t$ for the second half. We denote $[\phi_{ij}]$ to be the corresponding element in the universal cover $\widetilde{\text{Ham}}(M)$, i.e. the equivalence class of $\phi_{ij}$ under homotopy with fixed end-points. Note that a choice of two such non-trivial Hamiltonian functions determines all of them. Also note that $[\{\phi_{ji}^t\}] = [\{\phi_{ij}^{-1} \circ \phi_{ij}^{1-t}\}]$, $[\phi_{ij} \# \phi_{ji}] = \Id$, and $\overline{H}_{ij} = H_{ji} \circ \phi_{ij}^{-1}$. Finally, note that this diagram of Hamiltonian diffeomorphisms is indeed well defined as the following observation shows.
\begin{align*}
[\phi_{ji} \# \phi_{ik}] &= [\{\phi_{ji}^t\} \# \{ \phi_{ji}^1 \circ \phi_{ik}^t\}] \\
&= [ \{ \phi_{ij}^{-1} \phi_{ij}^{1-t} \} \# \{ \phi_{ij}^{-1} \circ \phi_{ij}^t \} \# \{ \phi_{ij}^{-1} \circ \phi_{ij}^1 \circ \phi_{jk}^t \}] \\
&= [\{ \phi_{jk}^t \}] = [\phi_{jk}].
\end{align*}

\begin{lemma}
\label{lemmaSpectralInvariantsLessThanSpectralNorm}
For a choice of $\{L_i\}$ and $\{H_{ij}\}$ as above, for every idempotent $a \in QH_*(L_1)$
denoting $a_i := (\phi_{i1}^1)_* a$,
there exists $i \neq j \neq k \in \{1,2,3\}$ with
\[ c(L_i;a_i,H_{ij}) + c(L_i;a_i,H_{jk} \circ \phi_{ji}^1) - c(L_i;a_i,H_{ik}) \leq \frac{1}{2} \nu(L_1;a,H_{12}) .\] 
\end{lemma}
\begin{proof}
Denote 
\[ A = c(L_1;a_1,H_{12}) + c(L_1;a_1,H_{23} \circ \phi_{21}^1) - c(L_1;a_1,H_{13}),\]
\[ B = c(L_2;a_2,H_{21}) + c(L_2;a_2,H_{13} \circ \phi_{12}^1) - c(L_2;a_2,H_{23}) .\]
We use the equation
\[ c(L;a,H) = c(\phi(L);\phi_* a,H \circ \phi^{-1}),\]
and the commutativity of the diagram $\{H_{ij}\}$ to get
\[ c(L_2;a_2,H_{23}) = c(L_1;a_1,H_{23} \circ \phi_{21}^1),\]
\[ c(L_2;a_2,H_{13} \circ \phi_{12}^1) = c(L_1;a_1,H_{13}),\]
\[ c(L_2;a_2,H_{21}) = c(L_1;a_1,H_{21} \circ \phi_{21}^1) = c(L_1;a_1,\overline{H}_{12}).\]
Hence
\[ A + B = c(L_1;a_1,H_{12}) + c(L_1;a_1,\overline{H}_{12}) = \nu(L_1;a,H_{12}).\]
We get that either $A \leq \frac{1}{2} \nu(L_1;a,H_{12})$, or $B \leq \frac{1}{2} \nu(L_1;a,H_{12})$.

\end{proof}

\begin{lemma}
\label{lemmaTriangleLessThanAction}
For a choice of $\{L_i\}$ and $\{H_{ij}\}$ as above, for any $y_{12} \in CF(L_1;H_{12}), y_{23} \in CF(L_1;H_{23} \circ \phi_{21}^1)$, there exists a pseudo holomorphic triangle $\Delta \in \mathcal{D}(L_1,L_2,L_3)$, with
\[ \area(\Delta) \leq \cA_{L_1;H_{12}}(y_{12}) + \cA_{L_1;H_{23} \circ \phi_{21}^1}(y_{23}) - \cA_{L_1;H_{13}}(y_{12} \ast y_{23}) .\]
\end{lemma}

\begin{proof}
Choose a critical point (i.e. Hamiltonian chord with capping) $(\alpha_{13},[\tilde{\alpha}_{13}])$ in $y_{12} \ast y_{23}$ with action $\cA_{L_1;H_{13}}(\alpha_{13},[\tilde{\alpha}_{13}]) = \cA_{L_1;H_{13}}(y_{12} \ast y_{23})$.
There are critical points $(\alpha_{12},[\tilde{\alpha}_{12}]) \in y_{12},(\alpha_{23},[\tilde{\alpha}_{23}]) \in y_{23}$ so that 
\[ \cA_{L_1;H_{12}}(\alpha_{12},[\tilde{\alpha}_{12}]) \leq \cA_{L_1;H_{12}}(y_{12}), \]
\[ \cA_{L_1;H_{23}\circ \phi_{21}^1}(\alpha_{23},[\tilde{\alpha}_{23}]) \leq \cA_{L_1;H_{23} \circ \phi_{21}^1}(y_{23}), \]
\[ (\alpha_{13},[\tilde{\alpha}_{13}]) \in (\alpha_{12},[\tilde{\alpha}_{12}]) \ast (\alpha_{23},[\tilde{\alpha}_{23}]). \]

From the equivalence between the Floer complexes $CF(L_i;H_{ij})$ and $CF(L_i,L_j)$ (see Section \ref{preliminariesSection}), we know that there are intersection points with cappings $(x_{12},[\tilde{x}_{12}]),$ $(x_{23},[\tilde{x}_{23}]),$ $(x_{13},[\tilde{x}_{13}])$ that match to $(\alpha_{12},[\tilde{\alpha}_{12}]),(\alpha_{23},[\tilde{\alpha}_{23}]),(\alpha_{13},[\tilde{\alpha}_{13}])$ respectively and have the same action and the same product. From the definition of the product, there exists a $J$-holomorphic triangle with symplectic area 
\[ \cA(x_{12},[\tilde{x}_{12}]) + \cA(x_{23},[\tilde{x}_{23}]) - \cA(x_{13},[\tilde{x}_{13}]) = \]
\[ \cA_{L_1;H_{12}}(\alpha_{12},[\tilde{\alpha}_{12}]) + \cA_{L_1;H_{23} \circ \phi_{21}^1}(\alpha_{23},[\tilde{\alpha}_{23}]) - \cA_{L_1;H_{13}}(\alpha_{13},[\tilde{\alpha}_{13}]) \leq \] 
\[\cA_{L_1;H_{12}}(y_{12}) + \cA_{L_1;H_{23} \circ \phi_{21}^1}(y_{23}) - \cA_{L_1;H_{13}}(y_{12} \ast y_{23}) .\]
This completes the proof.

\end{proof}

\begin{proof}[Proof of Theorem \ref{minTriangle}]
Let $L_1,L_2$ be as in the formulation of the theorem.
Fix an idempotent $a \in QH_* (L_1).$
Let $L_3$ be another Lagrangian Hamiltonian isotopic to $L_1$ that intersects $L_1,L_2$ transversely, as in the definition of $\cS(L_1,L_2)$.
Let $H:=H_{12}$ be a Hamiltonian generating a Hamiltonian flow $\phi_{12}$ with $\phi_{12}^1(L_2) = L_1$. Let $H_{13}$ be any Hamiltonian generating $\phi_{13}$ with $\phi_{13}^1(L_3) = L_1$. Complete $H_{12}$ and $H_{13}$ to a commutative diagram $\{H_{ij}\}$ as described in the beginning of Section \ref{holomorphicTrianglesSection}.
From Lemma \ref{lemmaSpectralInvariantsLessThanSpectralNorm}, there exists a choice of $i \neq j \neq k \in \{1,2,3\}$ with
\[ c(L_i;a_i,H_{ij}) + c(L_i;a_i,H_{jk} \circ \phi_{ji}^1) - c(L_i;a_i,H_{ik}) \leq \frac{1}{2} \nu(L;a,H) .\] 
Recall that $a_i:=(\phi_{i1}^1)_* a.$
Let $y_{ij} \in CF_*(L_i;H_{ij})$, $y_{jk} \in CF_*(L_i;H_{jk} \circ \phi_{ji}^1)$ be elements which are mapped under the PSS-isomorphism to $a_i$, and have actions $\cA_{L_i;H_{ij}}(y_{ij}) = c(L_i;a_i,H_{ij}), \cA_{L_i;H_{jk} \circ \phi_{ji}^1}(y_{jk}) = c(L_i;a_i,H_{jk}\circ \phi_{ji}^1).$
Note that $y_{ij} \ast y_{jk} \in CF_*(L_i;H_{ik})$ maps to $a_i \ast a_i = a_i$ under the PSS-isomorphism. Hence $\cA_{L_i;H_{ik}}(y_{ij} \ast y_{jk}) \geq c(L_i;a_i,H_{ik}).$ We get that
\[ \cA_{L_i;H_{ij}}(y_{ij}) + \cA_{L_i;H_{jk} \circ \phi_{ji}^1}(y_{jk}) - \cA_{L_i;H_{ik}}(y_{ij}\ast y_{jk}) \leq c(L_i;a_i,H_{ij}) + c(L_i;a_i,H_{jk}\circ \phi_{ji}^1) - c(L_i;a_i,H_{ik}).\]
Finally from Lemma \ref{lemmaTriangleLessThanAction}, there exists a pseudo-holomorphic triangle $\Delta \in \mathcal{D}(L_1,L_2,L_3)$, with
\[ \area(\Delta) \leq \cA_{L_i;H_{ij}}(y_{ij}) + \cA_{L_i;H_{jk} \circ \phi_{ji}^1}(y_{jk}) - \cA_{L_i;H_{ik}}(y_{ij}\ast y_{jk}).\]
Hence
\[ \inf \{\area(\Delta)\} \leq \frac{1}{2} \nu(L_1;a,H).\]
This is true for any $L_3$ and $H$ and hence
\[ \cS(L_1,L_2) \leq \frac{1}{2} \nu(L_1,L_2;a).\]

\end{proof}

\section{Triangles with small area in $S^2$}
\label{trianglesS2Section}

We consider $S^2$ equipped with the Euclidean area form $\omega$ normalized by $\int_{S^2} \omega = 1$.

We define \emph{equator} to be a closed curve in $S^2$ which is Hamiltonian isotopic to a great circle. A $median$ is an immersed curve $\gamma$
such that the area of each connected component of $S^2 \setminus \gamma$ is not greater than $1/2$. This way, equators are median curves. Other examples
include number eight figures with appropriate area constraints. We will always assume that a median has transverse self-intersections (if any).

Given three medians denoted by $L_R, L_G, L_B$ we will refer to them as the ``red'', ``green'' and ``blue'', respectively. We assume that they are in general position, that is, 
intersections are transverse and there are no triple intersection points.

Let $D \subset S^2$ be a closed disk, $\gamma : [a, b] \to D$ a curve. $\gamma$ is a \emph{chord} if $\gamma$ is simple, $\gamma ((a, b)) \subset \text{int}(D)$ and 
$\gamma(a), \gamma (b) \in \partial D$.
\begin{figure}
\centering
\includegraphics[width=0.5\textwidth]{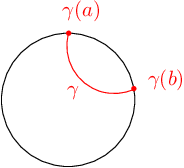}
\caption{A chord}
\end{figure}

\begin{proposition}
\label{s2triangleProp}
Let $L_R, L_G, L_B \subset S^2$ be three median lines in general position. Then there exists an immersed triangle $u : \Delta \to S^2$ whose sides 
lie on arcs of the medians and which has area $\area (u) = \int_\Delta u^*\omega \leq \frac{1}{8}$.
\end{proposition}

\begin{remark}
  When all the medians are equators, this proposition follows from Theorem \ref{minTriangle}. Here we give a different more direct proof of topological nature that also works in the case that some of the medians are figure eights.
\end{remark}

\begin{remark}
  When some medians are not equators, we allow sides of a triangle to be kinked at self-intersection points. 
\end{remark}
\begin{remark}
  The bound $1/8$ is sharp - it is obtained for three orthogonal great circles on $S^2$.
\end{remark}

\begin{proof}[Proof of Proposition \ref{s2triangleProp}] we analyze various configurations of arcs and intersection points and verify that they imply existence of 
triangles whose area is bounded by $\frac{1}{8}$. We also show that under assumptions of the proposition one of these configurations is guaranteed to arise.
In order to simplify the argument we will divide it into several steps. We will also assume first that the three medians are equators, and after we finish the argument 
for this special case we will explain modifications needed to extend it to medians in general.

The green equator $L_G$ divides $S^2$ into two disks of area $1/2$. Denote by $\Delta$ one of these disks.
$L_B$ restricted to $\Delta$ is a system of disjoint chords (see Figure \ref{chords_sys_fig}). The same holds for $L_R$. $L_B \cap L_R \neq \emptyset$
due to area considerations, hence there exists an intersection point either in $\Delta$ or in its complement.
Without loss of generality we assume that there exists such an intersection point $p \in L_B \cap L_R$ in $\Delta$.
Denote by $\gamma_R, \gamma_B$ the chords of $L_R \cap \Delta$, $L_B \cap \Delta$ which pass through $p$. Let
$a_R, b_R, a_B, b_B$ be the endpoints of $\gamma_R, \gamma_B$, respectively.

\begin{figure}
    \centering
    \includegraphics[width=0.5\textwidth]{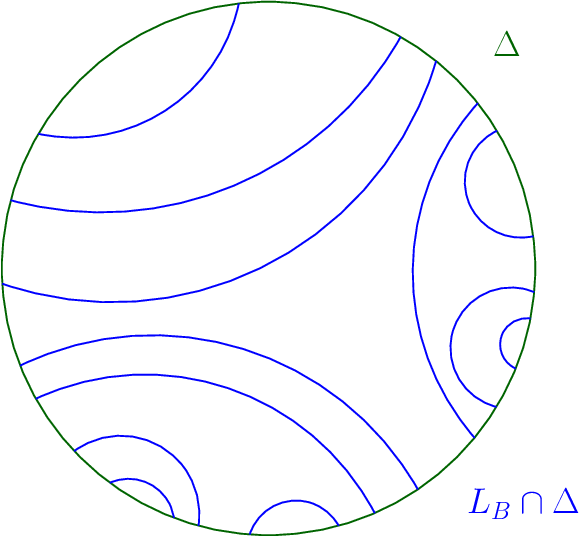}
    \caption{$L_B$ inside $\Delta$}
    \label{chords_sys_fig}
\end{figure}

\medskip

\underline{Step I:}
We say that the chords $\gamma_R, \gamma_B$ form a \emph{crossing} in $\Delta$ if the endpoints $a_R, b_R$ of $\gamma_R$ belong to
different connected components of $\Delta \setminus \gamma_B$ (see Figure \ref{crossing_fig}).

\begin{figure}
    \centering
    \includegraphics[width=0.5\textwidth]{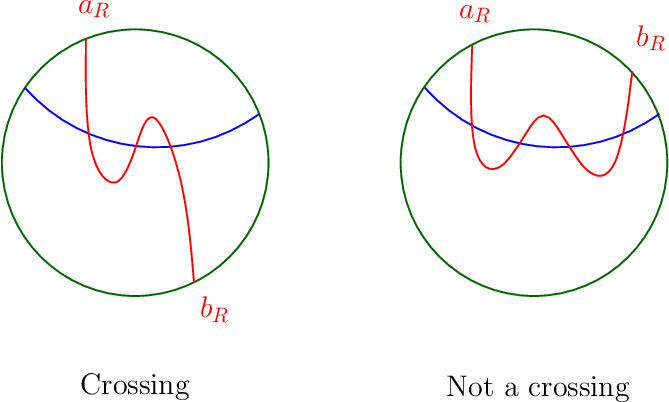}
    \caption{Examples for crossings}
    \label{crossing_fig}
\end{figure}

In the case of a crossing of $\gamma_R, \gamma_B$, there is an embedded triangle of area not greater than $\frac{1}{8}$:
$\gamma_B$ cuts $\Delta$ into two digons (we define digon to be a smooth embedding of a unit half-disk).
Orient $\gamma_R$ in direction from $a_R$ to $b_R$. Denote by $a_1$ its first intersection point with $\gamma_B$, by $b_1$ the last one.
(In the case when $p$ is the only intersection point we get $a_1=b_1=p$).

Now we note that the two arcs $[a_R, a_1], [b_1, b_R] \subset \gamma_R$ cut the two digons of $\Delta \setminus \gamma_B$ into four disjoint
triangles (see Figure \ref{crossing_triangles_fig}), hence the smallest has area at most $\area(\Delta)/4 = \frac{1}{8}$.

\begin{figure}
    \centering
    \includegraphics[width=0.5\textwidth]{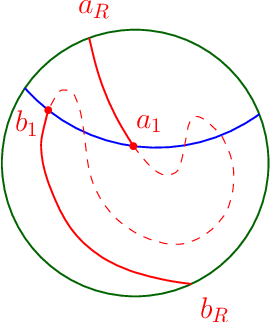}
    \caption{Arcs that cut $\Delta \setminus \gamma_B$ into four disjoint triangles}
    \label{crossing_triangles_fig}
\end{figure}

\begin{remark}
  When $\gamma_R, \gamma_B$ do not form a crossing, by a similar argument we can cut one of the digons (the one which contains both $a_R, b_R$) into two disjoint triangles.
  This way we immediately get an embedded triangle of area at most $\frac{1}{4}$. The remaining steps are to show that even in the case without crossings we can still obtain 
  the area bound of $\frac{1}{8}$ at the cost of considering immersed triangles (and not just embedded ones).
\end{remark}

\medskip

\underline{Step II:}
We assume there are no crossings (neither in $\Delta$, nor in its complement). This situation is possible: in Figure \ref{no_cross_fig} there are 
no crossings regardless of a choice of colors for the three equators.

\begin{figure}
    \centering
    \includegraphics[width=0.5\textwidth]{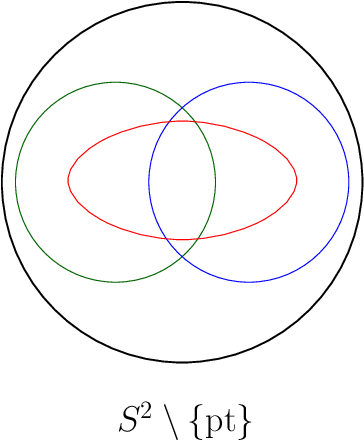}
    \caption{Example where there are no crossings}
    \label{no_cross_fig}
\end{figure}

Denote by $R$ the connected component of $\Delta \setminus L_B$ which contains $a_R$. Since there are no crossings, $b_R$ must belong to $R$ as well (see Figure \ref{region_fig}).
\begin{figure}
    \centering
    \includegraphics[width=0.5\textwidth]{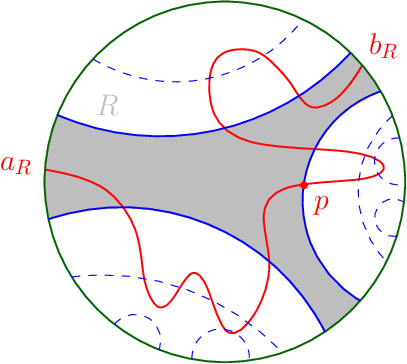}
    \caption{The connected component $B$ of $S^2 \setminus L_B$ that contains $a_R$ and $b_R$}
    \label{region_fig}
\end{figure}

$S^2 \setminus L_B$ consists of two disks. Denote by $B$ the one which contains the region $R$. In particular, both $a_R, b_R \in B$. 
A similar argument can be applied to each arc in $L_R \setminus L_G$, the endpoints must both lie either in $B$ or in its complement. Hence, each pair of 
consecutive (along $L_R$) intersections of $L_R \cap L_G$ must belong either to $B$ or to $S^2 \setminus B$. As $a_R \in B$ by definition, the same holds for all other 
intersection points: $L_R \cap L_G \subset B$.

$B$ may intersect $\Delta$ in a complicated way.  In order to simplify the setup, we lift the construction to the universal cover of $\Delta \cup B$:
consider a lift $\tilde{p}$ of $p \in L_R \cap L_B \cap \Delta$. Denote by $\widetilde{L_R}, \widetilde{L_B}, \widetilde{R}, \widetilde{\Delta}, \widetilde{B}$ the unique
lifts of the corresponding curves and regions which contain $\tilde{p}$ (see Figure \ref{lift_fig}). 

\begin{figure}
    \centering
    \includegraphics[width=0.5\textwidth]{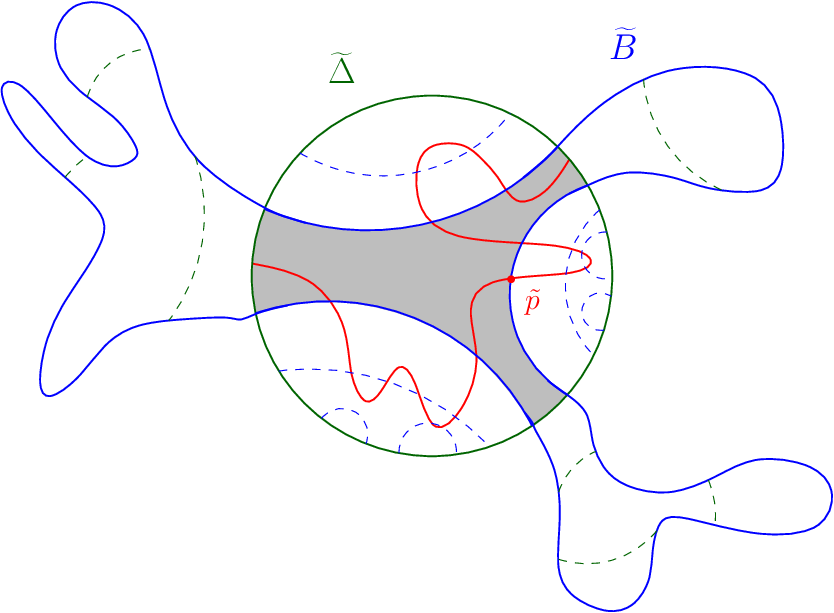}
    \caption{The universal cover of $B \cup \Delta$}
    \label{lift_fig}
\end{figure}

We claim that if we trace $\widetilde{L_R}$ starting from $\tilde{p}$ in either direction, we will never meet the boundary of $\widetilde{\Delta} \cup \widetilde{B}$. 
This implies $\widetilde{L_R}$ is not just an arc but a closed loop which is entirely contained in the interior of $\widetilde{\Delta} \cup \widetilde{B}$. 
Assume by contradiction that the claim is false. We restrict $\widetilde{L_R}$ to a chord in $\widetilde{\Delta} \cup \widetilde{B}$ which contains $\tilde{p}$. 
We continue to refer to this arc by the same $\widetilde{L_R}$.
The endpoints of $\widetilde{L_R}$ lie in $\partial (\widetilde{\Delta} \cup \widetilde{B})$. Since $L_R \cap \partial \Delta \subset \text{int} (B)$,
the endpoints belong to $\partial{\widetilde{B}} \setminus \widetilde{\Delta}$.
Hence the intersections $\widetilde{L_R} \cap \widetilde{L_B}$ lie both outside $\widetilde{\Delta}$ (e.g. the endpoints of $\widetilde{L_R}$) and inside (e.g. $\tilde{p}$).
Denote by $a, b$ two consequent (along $\widetilde{L_R}$) intersections such that $a \notin \widetilde{\Delta}$, $b \in \widetilde{\Delta}$ (see Figure \ref{red_lift_fig}).
Then in the disk $\widetilde{B}$ there is a crossing of the chord $[a, b] \subset \widetilde{L_R}$ and an arc of $\widetilde{L_G}$, which projects to a crossing in $B$
hence contradicts the initial assumption of Step II.

\begin{figure}
    \centering
    \includegraphics[width=0.5\textwidth]{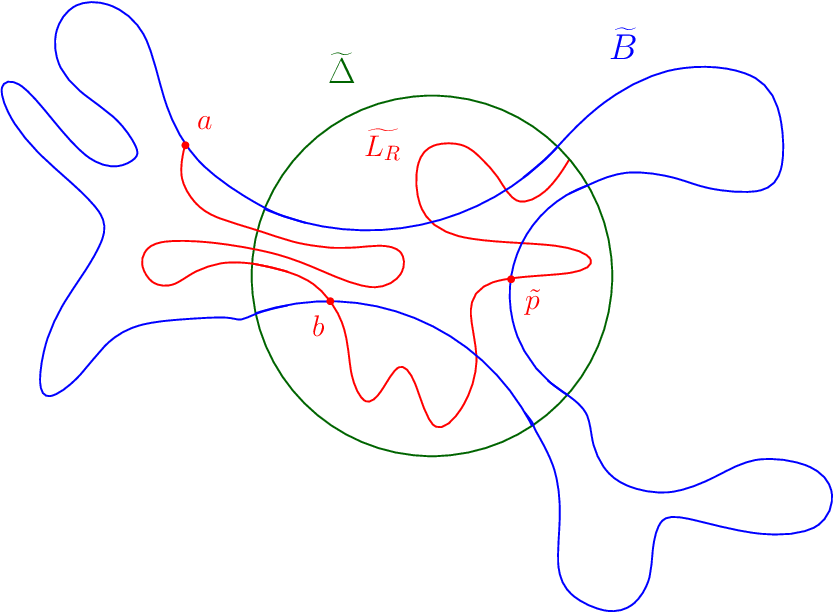}
    \caption{The arc $\widetilde{L}_R$}
    \label{red_lift_fig}
\end{figure}


Note that if one finds an immersed triangle in $\widetilde{\Delta} \cup \widetilde{B}$, it projects into $S^2$ to an immersed triangle of the same area. 
Therefore it is enough to show the claim in the covering space which has much simpler topology than $\Delta \cup B$. From now on we continue 
solely in the covering space. In order to shorten the notation we won't keep tildes attached to labels anymore.

Since the lifted $L_R \subset \Delta \cup B$, it implies $L_R \cap L_G \subset B \cap \Delta = R$, $L_R \cap L_B \subset \Delta \cap B = R$.
Hence all intersections of $L_R$ with the the other two equators occur in $\partial R$ (see Figure \ref{red_lift_int_fig}). Denote by $\Delta_R$ the disk bounded by $L_R$.

\begin{figure}
    \centering
    \includegraphics[width=0.5\textwidth]{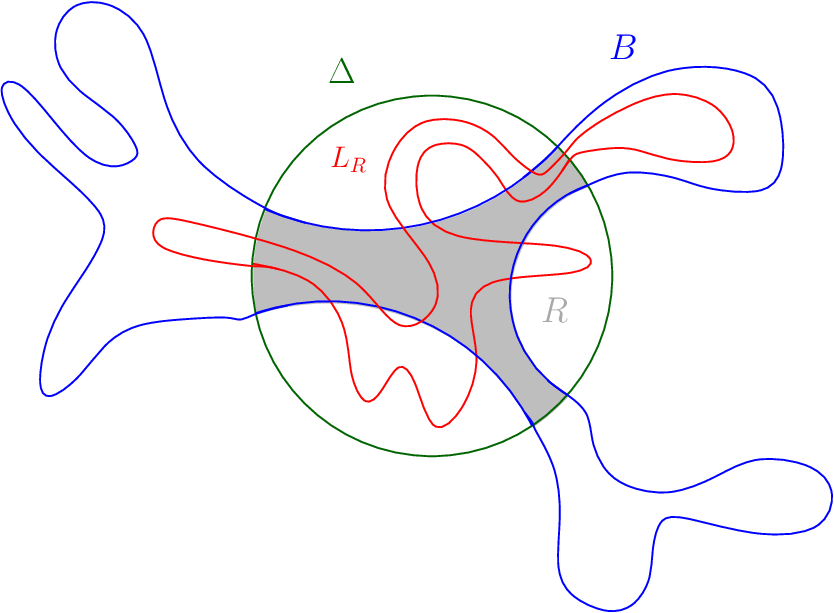}
    \caption{All of the intersections of $L_R$ with $L_B$ or with $L_G$ are in $\partial R$}
    \label{red_lift_int_fig}
\end{figure}

As the summary of the discussion above we obtain the disk $\Delta_R$ contained in the interior of $\Delta \cup B$ where the union $\Delta \cup B$
is simply connected.

\medskip

\underline{Step III:}
From now on we present $\Delta_R$ as a round disk and $L_G, L_B$ as a system of chords in the interior and the exterior of $\Delta_R$.
Note that $L_G, L_B$ (or $\Delta$, $B$) play the same role hence may be swapped in the arguments.

$L_B \setminus \Delta_R$ is a system of non-intersecting chords in the complement of $\Delta_R$. These chords may intersect with 
chords of $L_G \setminus \Delta_R$.
For a chord $\gamma \subseteq L_B \setminus \Delta_R$ (or a chord in $L_G \setminus \Delta_R$) define its \emph{base} $\alpha_\gamma$
to be the arc of $L_R$ which has the same endpoints as $\gamma$ and is homotopic relative endpoints to $\gamma$ in the complement of $\text{int}(\Delta_R)$ (see Figure \ref{chords_ext_fig}).

\begin{figure}
    \centering
    \includegraphics[width=0.5\textwidth]{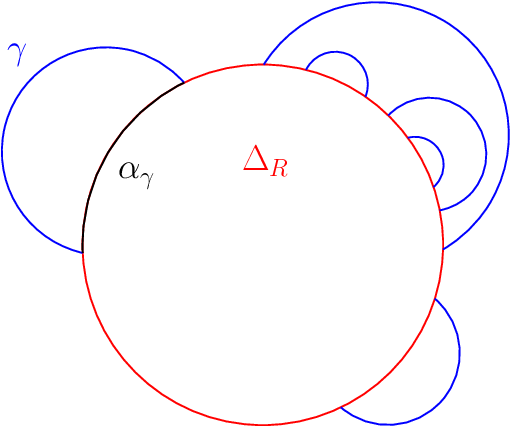}
    \caption{The chord $\gamma$ and its base $a_\gamma$}
    \label{chords_ext_fig}
\end{figure}

A chord is called not nested or \emph{essential} if its base is not contained in the base of any other chord. Clearly, the base of every chord
is contained in the base of some essential chord (see Figure \ref{chords_ess_fig}).

\begin{figure}
    \centering
    \includegraphics[width=0.5\textwidth]{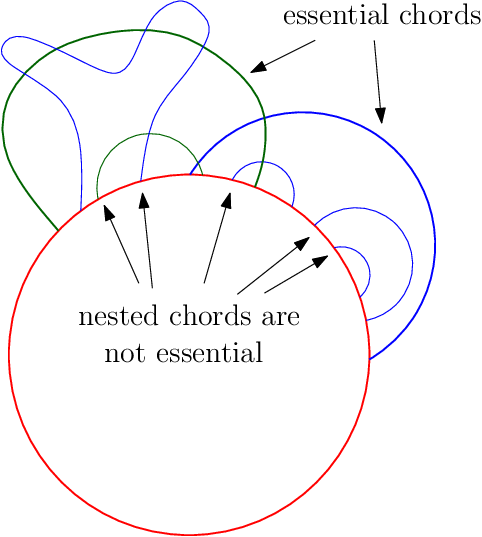}
    \caption{Essential and not essential chords}
    \label{chords_ess_fig}
\end{figure}

Since $L_R \subset \Delta \cup B$, any point $x \in L_R$ belongs to the base of some chord, and hence to the base of some essential chord. 
Therefore bases of essential chords provide partially overlapping (near ends) cover of $L_R$. As chords of the same color do not intersect,
``blue'' and ``green'' essential bases appear in alternating order.

The union of all essential chords encloses a region in $\R^2$. Naturally, this union is a subset of $\Delta \cup B$. No intersection of $L_B \cap L_G$
can occur in the interior of this region: $\Delta \cup B$ is simply connected and near each intersection point one of the four quadrants belongs to the 
complement of $\Delta \cup B$ (see Figure \ref{int_fig}). In particular, the entire region bounded by essential arcs is contained in $\Delta \cup B$.

\begin{figure}
    \centering
    \includegraphics[width=0.5\textwidth]{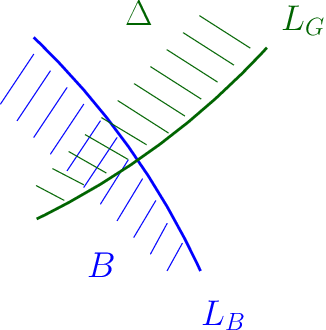}
    \caption{An intersection point of $L_B$ and $L_G$. One quadrant is outside $B \cup \Delta$}
    \label{int_fig}
\end{figure}

Note that each chord and its base bound a digon embedded in $(\Delta \cup B) \setminus \Delta_R$. If this chord is essential, the digon is cut by the 
subsequent (along $L_R$) essential chord into two disjoint triangles. Any two essential chords of the same color are disjoint, the same holds for their 
corresponding digons (see Figure \ref{chords_int_fig}). Hence if there are at least two ``blue'' essential chords, we have two disjoint digons (four disjoint triangles)
in $(\Delta \cup B) \setminus \Delta_R$.

\begin{figure}
    \centering
    \includegraphics[width=0.5\textwidth]{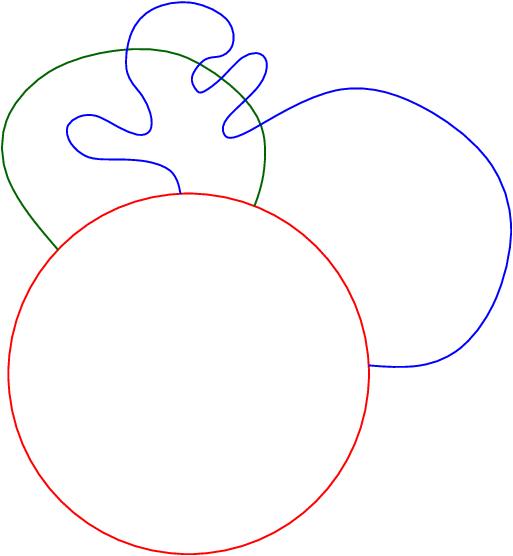}
    \caption{Examples of green and blue adjacent digons where each digon is cut into two triangles}
    \label{chords_int_fig}
\end{figure}

\[
  \begin{array}{rcl}\area((\Delta \cup B) \setminus \Delta_R) &=& \area(\Delta) + \area(B) - \area (\Delta \cap B) - \area (\Delta_R) \\
		&=& \frac{1}{2} + \frac{1}{2} - \area (\Delta \cap B) -\frac{1}{2} < \frac{1}{2}. 
  \end{array}
\]
It follows that there is a triangle of area less than $\frac{1/2}{4} = \frac{1}{8}$.

The same argument holds for the ``green'' essential chords. Hence the only case left to consider is when there is a single blue and a single green essential chord.
Another observation: if $\area (\Delta \cap B) \geq \frac{1}{4}$, $\area((\Delta \cup B) \setminus \Delta_R) \leq \frac{1}{4}$. In this case it is sufficient to 
have a single digon cut into two triangles in order to satisfy the claim.

\begin{figure}
    \centering
    \includegraphics[width=0.5\textwidth]{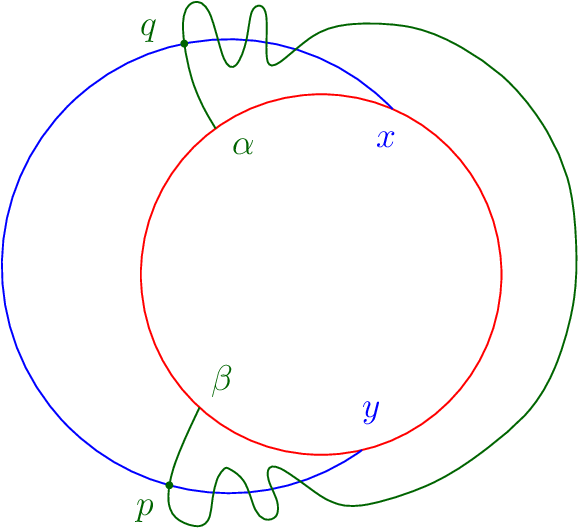}
    \caption{The two essential chords with their endpoints}
    \label{chords_ess_int_fig}
\end{figure}

So we may assume that $\area (\Delta \cap B) < \frac{1}{4}$.

\begin{remark}
  the triangles constructed in this step are embedded in $\Delta \cup B$ but not necessarily lie entirely in $\Delta$ or entirely in $B$. Therefore after 
  projection to $S^2$ they may become immersed.
\end{remark}

\underline{Step IV:}
We continue with the assumptions of the last part of Step III. Denote by $\alpha, \beta$ the endpoints of the green essential chord $\gamma_G$,
by $x, y$ the endpoints of the blue essential chord $\gamma_B$ as in Figure \ref{chords_ess_int_fig}.
Orient $\gamma_G$ from $\alpha$ to $\beta$. Let $q$ be the first and $p$ be the last intersection of $\gamma_G \cap \gamma_B$ (ordered along $\gamma_G$).

In the case there are additional intersection points of the two essential chords (except for $q, p$), they must appear either on $[p, y] \subset \gamma_B$ so that their order along $\gamma_B$ 
coincides with that along $\gamma_G$ or on $[q, x] \subset \gamma_B$ and the orders agree as well. Otherwise there will be intersection points ``trapped'' inside $\gamma_B \cup \gamma_G$ in 
the contrary to simple connectedness of $\Delta \cup B$. In the case such additional intersection points do appear, we call the digons 
bounded by $\gamma_G$ outside $\gamma_B$ \emph{ears} (see Figure \ref{ears_fig}). Once again, since $\Delta \cup B$ is simply connected, no intersections of $L_G \cap L_B$ can occur in
the interior of the region bounded by $\gamma_B, \gamma_R$ (including ``ears'' and $\Delta_R$).

\begin{figure}
    \centering
    \includegraphics[width=0.5\textwidth]{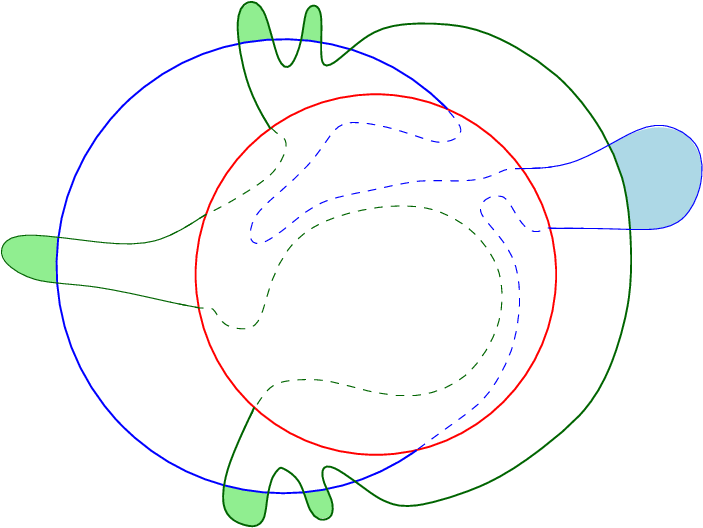}
    \caption{Examples of ears}
    \label{ears_fig}
\end{figure}

We also call an ``ear'' a region trapped between $\gamma_G$ and a non-essential blue chord whose base is nested in the base of $\gamma_G$. 
Such regions (if exist) lie outside $\gamma_G$. In a similar way we call an ear a region trapped between a non-essectial chord of $L_G$ and $\gamma_B$ and which 
touches $\gamma_B$ from outside. Note that all ``ears'' lie outside the digon formed by $\gamma_B$. 

Now we can refine our area estimate: $$\text{digon}(\gamma_B) \subseteq (\Delta \cup B) \setminus \Delta_R \setminus \bigcup \text{ears}.$$ Hence
if $1/2 - \area(\Delta \cap B) - \area (\bigcup \text{ears}) \leq \frac{1}{4}$ $\Leftrightarrow$
$\area(\Delta \cap B) + \area (\bigcup \text{ears}) \geq \frac{1}{4}$, one of the triangles inside the digon of $\gamma_B$ has area at most $\frac{1}{8}$.

\medskip

We assume the contrary, namely, $\area(\Delta \cap B) + \area (\bigcup \text{ears}) < \frac{1}{4}$.

As in Figure \ref{chords_ess_int2_fig}, denote by $p'$ the nearest to $y$ (along $\gamma_B$) intersection point of $\gamma_B \cap \gamma_G$, by $q'$ the same for $x$ ($p', q'$ may coincide with $p, q$).
Consider the triangle formed by $p', \beta, y$. If the arc $[\beta, y] \subset L_R$ has no intersections with $L_G, L_B$, this triangle is contained in $(B \cap \Delta) \cup \bigcup \text{ears}$.
Otherwise, if there exist intersections with $L_B$, extend $[p', y] \subset L_B$ beyond $y$ along $L_B$ till the intersection point $y' \in [\beta, y]$ which is the closest
along $L_R$ to $\beta$. Do the same for $[p', \beta] \subset L_G$ till the intersection $\beta'$ closest to $y'$ in $[\beta, y']$ (see Figure \ref{tria_fig}).

\begin{figure}
    \centering
    \includegraphics[width=0.5\textwidth]{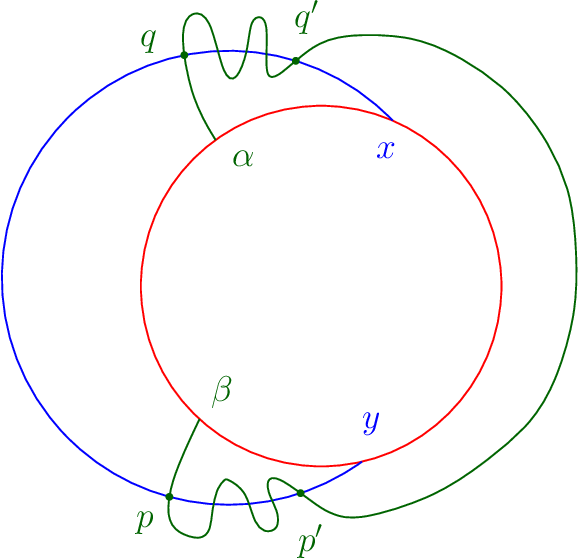}
    \caption{The interstion points $p,p',q,q'$}
    \label{chords_ess_int2_fig}
\end{figure}

\begin{figure}
    \centering
    \includegraphics[width=0.5\textwidth]{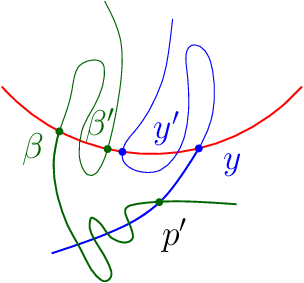}
    \caption{Extension of $y$ and $\beta$ to $y'$ and $\beta'$ respectively}
    \label{tria_fig}
\end{figure}

By this construction, the red side of the triangle $p'y'\beta'$ does not intersect any curve. The green and the blue sides do not intersect one another, hence the triangle is embedded.
However, the green and the blue sides may intersect the essential chords of the oppposite color ($\gamma_B, \gamma_G$, resp.)
In the case of such intersections, the triangle is not contained in $\Delta \cap B$, but the ``extra'' parts consist of ears (see Figure \ref{tria2_fig}).

\begin{figure}
    \centering
    \includegraphics[width=0.5\textwidth]{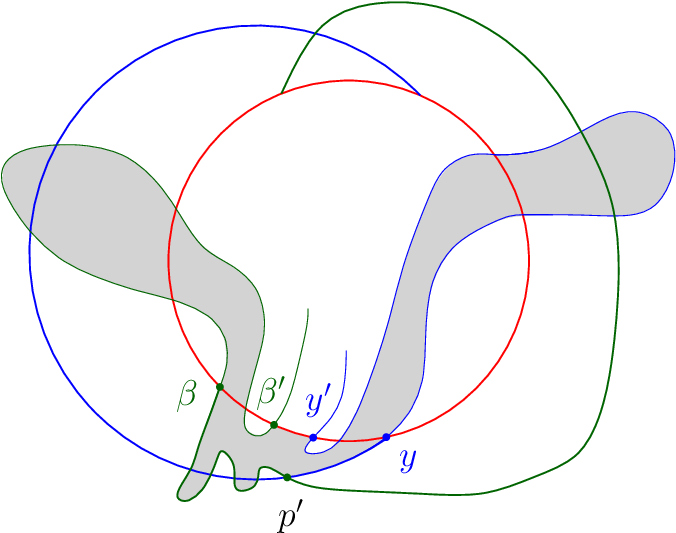}
    \caption{the triangle $p'y' \beta'$ lies inside $\Delta \cup B \cup$ ears}
    \label{tria2_fig}
\end{figure}

Similarly, we construct the symmetric triangle $q'\alpha'x'$. The sides of these two triangles do not overlap and do not intersect, hence they are disjoint. They are contained
in the region $(\Delta \cap B) \cup \bigcup \text{ears}$ which has area less than $\frac{1}{4}$, therefore one of them has area less than $\frac{1}{8}$.

\medskip

\underline{Step V:} We drop the assumption that $L_i$ are equators and consider general median curves. We review the preceding four steps and describe modifications necessary to support the more general setup. 
We note that any two medians intersect due to area considerations in the same way as equators do. 

Step I: denote by $\Delta$ one of the connected components of $S^2 \setminus L_G$ which contains an intersection of $L_R \cap L_B$. As $\area(\Delta) \leq 1/2$, we can use the same argument as in Step I
and continue assuming there are no crossing of any two medians.

Step II: 
$S^2 \setminus L_B$ may consist of several disks. We pick that which contains $R$ and continue with the argument verbatim. 
At the end we lift the regions $R, \Delta, B$ and the medians to the universal cover and conclude that $\widetilde{L_R}$ is contained in the union of the lifted disks $\Delta \cup B$.
Note that both disks $\Delta, B$ may have areas less than $1/2$ in the non-equator case. However $L_R$ is considered as a whole, so in the case it forms a figure eight it bounds area greater than $1/2$.

Step III: if $L_R$ is not embedded, we denote by $\Delta_R$ the union of bounded connected components of its complement. Note that $\Delta_R$ is a family of topological disks glued one to another at a boundary point, or (if we allow for degenerate cases) the topological disks might be glued in several points and along segments, but either way, for any $\varepsilon > 0$, the $\varepsilon$-neighborhood of $\Delta_R$ is a topological disk. 
Hence, since $L_B, L_G$ cannot go through self-intersections of $L_R$, we may effectively replace $\Delta_R$ by its $\varepsilon$-neighborhood and continue the argument as before.
This way we obtain $\area(\Delta), \area(B) \leq 1/2$, $\area(\Delta_R) \geq 1/2$ and the same inequality holds - it may even get improved.

Step IV: taken verbatim.

\end{proof}

\section{The median quasi-state} \label{medianSection}

In this section we prove Theorem \ref{median_theorem}.
First we prove that for $f,g \in C^\infty(S^2)$, 
$$ \frac{| \zeta(f+g) - \zeta(f) - \zeta(g) |^2}{\|\{f,g\}\|} \leq \frac{1}{4} .$$
Then we show an example for $f$ and $g$ such that 
$$ \frac{| \zeta(f+g) - \zeta(f) - \zeta(g) |^2}{\|\{f,g\}\|} \geq \frac{1}{4} - \delta $$
for every $\delta > 0.$

\subsection{Bound on the median quasi-state}
Let $f,g \in C^\infty(S^2)$. Consider the map $\Phi: S^2 \to \R^2$, $\Phi(z) = (f(z),g(z))$.
Let $D = \{(x,y) : x> \zeta(f), y > \zeta(g), x + y < \zeta(f+g)\}$. Note that the (Euclidean) area $A$ of the triangle $D$ satisfies
$$
A = \frac{1}{2} | \zeta(f+g) - \zeta(f) - \zeta(g) |^2.
$$
Consider the medians $m_f,m_g,m_{f+g}$ of $f,g,f+g$ respectively. Every triangle $\Delta \subset S^2$ with boundary on $m_f \cup m_g \cup m_{f+g}$ satisfies that $D \subset \Phi(\Delta)$.
Now
\[ A \leq \int_{\Phi^{-1}(D)} |df \wedge dg| = \int_{\Phi^{-1}(D)} |\{f,g\} \omega| \leq \int_\Delta |\{f,g\} \omega| \leq \area(\Delta) \| \{f,g\} \|. \]
From Proposition \ref{s2triangleProp} there exists a triangle $\Delta$ with $\area(\Delta) \leq \frac{1}{8}$ and hence
\[ \frac{A}{\|\{f,g\}\|} \leq \frac{1}{8}, \]
which implies
\[ \frac{| \zeta(f+g) - \zeta(f) - \zeta(g) |^2}{\|\{f,g\}\|} \leq \frac{1}{4}. \]

\subsection{Sharpness example}
Let $\Sigma$ be a smoothing of the set of 4 triangles of area $\frac{1}{8}$ as described in Figure \ref{triangleFigure}.

\begin{figure}
    \centering
	\includegraphics[width=0.5\textwidth]{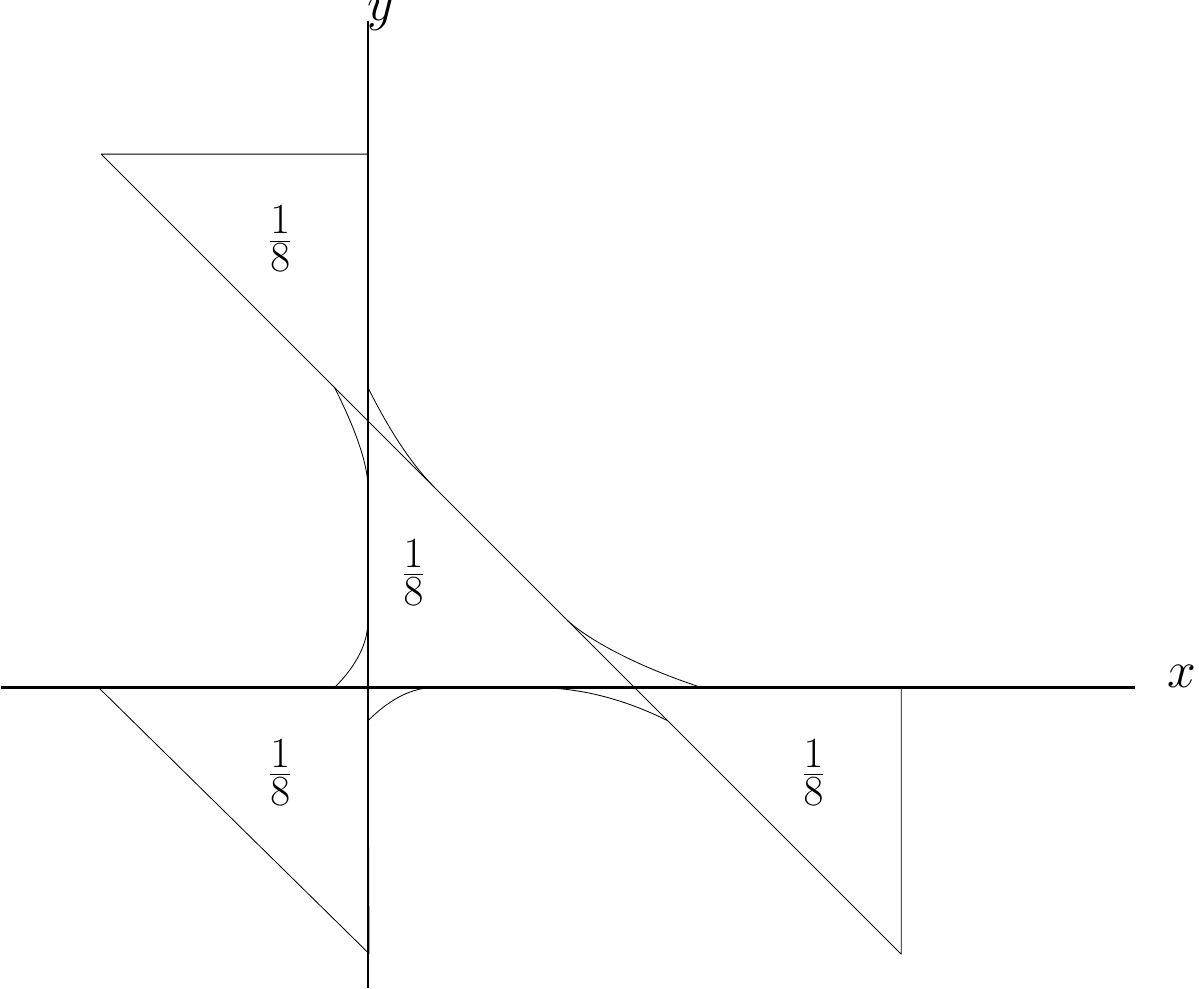}
	\caption{The set $\Sigma$}
	\label{triangleFigure}
\end{figure}

Let $\phi : D^2 \to \Sigma$ be an area preserving map, where $D^2$ is the standard disc of area $1$ and with radius $R = \frac{1}{\sqrt{\pi}}$.
Define the functions $f,g : D^2 \to \R$ by $f(z)$ ($g(z)$) equals the $x$-coordinate ($y$-coordinate) of $\phi(z)$.
One has $\{f,g\} \equiv 1$. 
By gluing two copies of $D^2$ with opposite orientations, define the functions $f_s,g_s : S^2 \to \R$ by taking the values of $f$ and $g$ on each hemisphere. Note that $f_s$ and $g_s$ are not differentiable on the equator.
However, $\{f_s,g_s\} \equiv 1$ on the upper hemisphere outside a neighborhood of the equator, and $\{f_s,g_s\} \equiv -1$ in the lower hemisphere outside a neighborhood of the equator.
Note also that $\zeta(f_s)$ and $\zeta(g_s)$ are arbitrarily close to $0$ and $\zeta(f_s+g_s)$ is arbitrarily close to $\frac{1}{2}$.

Let $h : D^2 \to \R$ be a smooth function. For $z = r e^{i\theta}$, define 
\[\tilde{h}(r,\theta) = h(0) + \int_0^r a(t) \partial_{r} h(t ,\theta) dt,\]
where $a: [0,R] \to \R$ is a cut-off function that satisfies $a |_{[0,R-2\epsilon]} \equiv 1$ and $a |_{[R-\epsilon,R]} \equiv 0$.
For small enough $\epsilon$, $\tilde{h}$ is a $C^0$ approximation of $h$ that satisfies $\partial_r h \equiv 0$ in a neighborhood of the boundary, and moreover one can check that $\| \{\tilde{f},\tilde{g} \}\| \leq \| \{f,g\} \|  + \varepsilon$ for $\varepsilon > 0$ arbitrarily small.


Since the $r$-derivative near the boundary is $0$, one gets that $\tilde{f}_s$ and $\tilde{g}_s$ are smooth functions from $S^2$ to $\R$ that satisfy $\| \{ \tilde{f}_s ,  \tilde{g}_s\}\| \leq 1 + \varepsilon$. Since $\tilde{f}_s$ and $\tilde{g}_s$ are $C^0$ approximations of $f_s$ and $g_s$ respectively, one has that $\zeta(\tilde{f}_s)$ and $\zeta(\tilde{g}_s)$ are still arbitrarily close to $0$ and $\zeta(\tilde{f}_s+\tilde{g}_s)$ is arbitrarily close to $\frac{1}{2}$.
Hence for some arbitrarily small $\delta >0$ we get an example where
\[ | \zeta(\tilde{f}_s+\tilde{g}_s) - \zeta(\tilde{f}_s) - \zeta(\tilde{g}_s) |^2 \geq (\frac{1}{4} - \delta) \| \{\tilde{f}_s,\tilde{g}_s\} \| .\]
This proves sharpness.

\bibliography{references}
\bibliographystyle{siam}

\end{document}